\documentclass[12pt,draft]{article}
\usepackage[cp1251]{inputenc}
\usepackage{amsthm,amssymb,amsmath,latexsym}
\usepackage[mathscr]{eucal}
\usepackage{indentfirst}
\usepackage{enumerate}
\usepackage{verbatim}
\theoremstyle{plain}
\newtheorem{theorem}{Theorem}
\newtheorem*{theom}{Main Theorem}
\newtheorem*{theoBM}{Beurling--Malliavin Theorem}
\newtheorem*{corollary}{Corollary}

\theoremstyle{definition}

\newtheorem{example}{Example}

\theoremstyle{remark}
\newtheorem{remark}{Remark}

\newcommand{\R}{{\mathbb R}}
\newcommand{\bC}{{\mathbb C}}
\newcommand{\N}{\mathbb N}
\newcommand{\Z}{{\mathbb Z}}
\newcommand{\e}{\varepsilon}
\newcommand{\pvint}{{\not}{\mspace{-2mu}\int}}
\newcommand{\mc}{\mathcal}
\renewcommand{\d}{{\,d}}
\renewcommand{\phi}{\varphi}
\newcommand{\Exp}{{\mathrm{Exp}}}
\renewcommand{\Re}{{\mathrm{Re}\,}}
\renewcommand{\Im}{{\mathrm{Im}\,}}
\DeclareMathOperator{\Poi}{P}
\DeclareMathOperator{\exc}{exc}
\DeclareMathOperator{\Hil}{H}

\author{ Bulat N.~Khabibullin\footnote{Supported by Russian Foundation of Basic Reseach, grant No.~09-01-00046-a}}
\title{Distribution of zero subsequences for
Bernstein space  and criteria of
 completeness\\ for exponential system on a segment}
\date{}

\begin{document}
\maketitle

{\abstract{For $\sigma\in (0,+\infty)$, denote by
 $B_\sigma^{\infty}$  the Bernstein space
(of type $\sigma$) of all entire  functions of
exponential type $\leq \sigma$ bounded on real axis $\R$. Let $I_d\subset \R$ be a segment of length $d>0$. We announce
complete description of non-uniqueness sequences of points for $B_\sigma^\infty$ and criteria of completeness of exponential system in $C(I_d)$  or $L^p(I_d)$
accurate within one or two exponential functions.

\noindent{\bf Keywords:}{ entire function, Bernstein space, zero subsequence, uniqueness sequence, completeness, exponential system, Poisson integral, Hilbert transform}}}

\section{Definitions, and statements of problems}
 Denote by $\N$,  $\R$, $\Z$, and  $\bC$
the sets of natural, real, integer, and complex
numbers respectively. Besides, $\R_*:=\R \setminus \{0\}$, 
$\bC_*:=\bC \setminus \{0\}$, $\bC_\pm:=\bC\setminus \R$.

For a segment $I_d\subset \R$ of length $d$, we denote by  $C(I_d)$ and
$L^{p} (I_d)$  the space of continuous functions $f$ on
$I_d$ with $\sup$-norm $\|f\|_\infty:=\sup \bigl\{ |f(x)|\colon x\in I_d\bigr\}$ and the space of functions
 $f$ with finite norm
\begin{equation*}
\|f\|_{p}:=\left(\int_{I_d} | f(x) |^p \d x
\right)^{1/p},
\quad p\geq 1.
\end{equation*}
For $\sigma\in (0,+\infty)$, denote by
 $B_\sigma^{\infty}$  {\it the Bernstein space\/}
(of type $\sigma$) of all entire functions of
exponential type $\leq \sigma$ bounded on $\R$, i.\,e.
\begin{equation}\label{df:Bs}
		\log|f(z)|\leq \sigma|\Im z|+c_f,
	\quad  z\in \bC,
\end{equation}
where $c_f$ is a constant (see \cite{Lev96}).

Let 
\begin{equation}\label{df:L}
\Lambda =\{ \lambda_k \}_{k\in \N}\subset \bC
\end{equation}
be a point sequence on $\bC$ without limit points
in $\bC$.  Each sequence $\Lambda$ defines
a {\it counting measure\/}
\begin{equation}\label{df:nLa}
n_{\Lambda}(S):=\sum_{\lambda_k \in S} 1,
\quad S\subset \bC ,
\end{equation} 
The sequence $\Lambda$ is a {\it zero subsequence\/} (non-uniqueness se\-qu\-en\-ce) for $B_\sigma^\infty$ iff there exists a nonzero function $f_\Lambda\in B_\sigma^\infty$ such that $f_\Lambda$ vanish on $\Lambda$, i.\,e. $f_\Lambda^{(m-1)} (\lambda_k)=0$ for  all $k\in \N$ and $1\leq m \leq n_\Lambda \bigl(\{\lambda \}\bigr)$.  The sequence $\Lambda$
defines also the {\it exponential system}
\begin{equation}\label{df:Exp}
\Exp^\Lambda {:=}\{x\mapsto
x^{m-1}e^{\lambda x} \colon x\in I_d,
\; \lambda \in \Lambda ,
\; 1\leq m\leq n_\Lambda \bigl(\{\lambda \}\bigr),
\; m\in \N  \}.
\end{equation}

The exponential system 
$\Exp^{i\Lambda}$, $i\Lambda:=\{ i\lambda_k\}_{k\in \N}$,
 is {\it complete\/} in $C(I_d)$  or $L^p(I_d)$  iff the closure of its linear span coincides with $C(I_d)$  or $L^p(I_d)$.

We solve the following problems.
\begin{itemize}
\item {\it Complete description of zero subsequences for $B_\sigma^\infty$.}
\item {\it Criteria of completeness of exponential system\/ $\Exp^{i\Lambda}$ in $C(I_d)$  or $L^p(I_d)$
accurate within one or two exponential functions.}
\end{itemize}
It was many times noticed that obtaining of criteria of completeness of exponential systems $\Exp^{i\Lambda}$ in $C(I_d)$  or $L^p(I_d)$ purely in terms $\Lambda$ and $d$ is exclusively difficult problem 
\cite{Kras86}, \cite{KhNP}, \cite[p.~167]{Koosis92}, \cite[3.3]{HMN04}, \cite[0.1]{MP05}.
 
\paragraph{Poisson integral.}
For $z\in \bC_{\pm} $ 
{\it the Poisson integral\/}
 $\Poi_{{\bC_{\pm}}}\, \phi$ of function 
$\phi \in L^1(\R )$ is defined as
\begin{align}\label{int:P}
(\Poi_{ \bC_\pm} \phi )(z):&=
\frac1{\pi}\int_{-\infty}^{+\infty}
\frac{|\Im z|}{(t-\Re z)^2
+(\Im z)^2} \, \phi (t)\d t \notag\\
&=\frac1{\pi}\int_{-\infty}^{+\infty}
\left|\Im\frac{1}{t- z} \right|\, \phi (t)\d t, 	
\end{align}
But for $x \in \R$
we set
\begin{equation}
(\Poi_{{ \bC_\pm}} \phi)(x):= 
\phi (x),
\quad x \in \R_*.
\end{equation}
The Poisson integral from \eqref{int:P} is harmonic function on $\bC_\pm$.

If $\phi \in C(\R_*)$, then the Poisson transform $\Poi_{{ \bC_\pm}} \phi$ is also continuous function on $\bC_*$. More in detail look in \cite{HJ}.

\paragraph{Hilbert transform ({\rm see \cite{HJ}}).}
For a function 
$$
\phi \in L^1(\R_*),  \quad\phi \colon \R_*\to \R,
$$ 
the {\it direct Hilbert transform\/} $\Hil$ is defined as a rule by integral
\begin{equation*}
(\Hil \phi)(x):=\frac{1}{\pi}\;\pvint_{\R_*}\frac{\phi (t)}{x-t} \d t, \quad x\in \R_*,
\end{equation*}
where the strikethrough of integral 
$$\pvint_{\R_*}:=\lim_{0<\e\to 0}
\int_{\R_* \setminus  (x-\e,  x+\e)} \, , \quad x\neq 0,$$ 
means the principal integral value  in the sense of Cauchy.
Such function $\Hil \phi$ is good defined almost everywhere on $\R$.  The {\it inverse  Hilbert transform\/} differs only by sign
\begin{equation*}
\hspace{-2mm}(\Hil^{-1} \phi)(x):=\frac{1}{\pi}\;\pvint_{\R_*}\frac{\phi (t)}{t-x} \d t=-(\Hil \phi)(x), \quad x\in \R_*.
\end{equation*}	

\section{Main Theorem}
\setcounter{equation}{0}
\paragraph{Classes of test functions $R\mc P_0^{m}$,
$2\leq m\leq \infty$.} This class will consist of all
{\it positive continuous functions}
$$
\phi \colon \R_* \to [0,+\infty)	, \quad Z_\phi:=\left\{x\in \R_*\colon \phi (x)=0\right\},
$$
from the class
 $C^m(\R_* \setminus Z_\phi)$ of
$m$ times continuously differentiable functions
on $\R_*\setminus   Z_\phi$ with
\begin{itemize}
\item {\it a finiteness condition}
\begin{equation}\label{co:phin}
\phi (x)\equiv 0, \quad  |x|\geq R_\phi>0,
\end{equation}
where $R_\phi>0$ is a constant;
\item {\it a semi-normalization condition\/}
\begin{equation}\label{co:no}
\limsup_{0\neq x\to 0} \frac{\phi (x)}{\log \bigl(1/{|x|}\bigr)}\leq 1;
\end{equation}

\item
{\it a conjugate  condition of positivity}
\begin{equation}\label{co:dpos}
(-\Hil \phi)'(x):=\frac{1}{\pi}\;
\pvint_{\R_*}
\frac{\phi (t)-\phi (x)}{(t-x)^2} \d t\geq 0,
\quad x\in \R_*\setminus  Z_\phi,
\end{equation}
i.\,e. condition of decrease for the   Hilbert
transform $\Hil \phi$ separately
on every connected component (open interval) of the  subset
$\R_*\setminus   Z_\phi \subset \R$.
\end{itemize}
The Hilbert transform and the differentiation operator commute, i.\,e.
\begin{equation*}
	\frac{d}{\d x}(-\Hil \phi)\equiv
 -\Hil \frac{d}{\d x}\phi, 
\end{equation*}
 for $\phi \in L^p ( \R_*)$, $p\geq 1$, and for Schwartz distributions $\phi$ (see \cite{Pan96}).

Besides,  the  left-hand member of \eqref{co:dpos} can be rewritten as
\begin{equation*}
(-\Hil \phi )'(x)=\frac{1}{\pi}\;
\pvint_{0}^{+\infty}
\frac{\phi(x+t)+\phi(x-t)-2\phi(x)}{t^2} \d t .
\end{equation*}

\begin{theom}[\rm on zero
 subsequence for Bernstein space]
Let $\Lambda 
\not\ni 0$
 be a point sequence on $\bC$ from \eqref{df:L}, and
  $\sigma \in (0,+\infty)$.
The following three assertions are equivalent:
\begin{enumerate}[\rm 1)]
\item $\Lambda$ is
zero subsequence for Bernstein space
 $B_\sigma^{\infty}$;
\item for the some (for any) $m\in \bigl(\N \setminus \{1\}\bigr)\cup \infty$
the condition
\begin{equation}\label{est:t}
\sup_{\phi\in R\mc P_0^{m}}
\left( \sum_{k\in \N} (\mathrm P_{{\bC_\pm}} \phi)(\lambda_k)
-\frac{\sigma}{\pi}\int_{-\infty}^{+\infty}\phi (x) \d x
\right)	<+\infty 
\end{equation}
 is fulfilled.
\item the condition of\/ \eqref{est:t} is fulfilled after
replacement of the class
 $R\mathcal P_0^{m}$ with $R\mathcal P_0^{m}\cap C^\infty (\R_*)$, where $C^\infty(\R_*)$ is the class of infinitely differentiable  functions on $\R_*$.
\end{enumerate}
\end{theom}

\begin{remark} If $\Lambda \subset
\R_*$ is real sequence, then \eqref{est:t}
looks absolutely simply:
\begin{equation}\label{co:Lre}
\sup_{\phi \in \mathcal R\mc P_0^m}\left(
\sum_{\lambda\in \Lambda}	
\phi(\lambda)-\frac{\sigma}{\pi}
\int_{-\infty}^{+\infty}\phi (x) \d x
\right)<+\infty.
\end{equation}
\end{remark}

\begin{remark}\label{r:2}
Any shifts of any finite number of points from
$\Lambda$ do not change his property as
zero subsequence for $B_\sigma^\infty$. Therefore, without loss of generality, we can consider only cases
$0\notin \Lambda$.
\end{remark}

\begin{remark} Conditions \eqref{est:t}and \eqref{co:Lre} are similar to definition of order relation  on sets of real-valued Radon measures or on distributions. Indeed, if $\nu$ and $\mu$ are two such measures or two distributions on $\bC$, then $\nu \leq \mu$ iff 
\begin{equation}\label{df:ord}
	\sup_{\phi \in (C_0^\infty(\bC))^+
	} \bigl(\nu (\phi)-\mu (\phi)\bigr)\leq 0,
\end{equation}
where $(C_0^\infty(\bC))^+$ is the class of infinitely differentiable finite positive functions on $\bC$.

The measure $\mu_\sigma$ with density 
$$
\d \mu_\sigma(x):=\frac{\sigma}{\pi}\d x, \quad x\in \R,
$$ 
on $\R$ is the Riesz measure of subharmonic function 
$$M_\sigma (z):=\sigma |\Im z|, \quad z\in \bC, 
\quad \mu_\sigma:=\frac1{2\pi} \, \Delta M_\sigma \geq 0,
$$ 
where $\Delta$ is the Laplace operator in the sense of the Schwartz distribution theory.
The function $M_\sigma$ is located in the right-hand mem\-ber of the inequality \eqref{df:Bs} and  defines the Bernstein space $B_\sigma^\infty$. In these notations the conditions \eqref{est:t} and \eqref{co:Lre} it is possible to write in the form (see \eqref{df:nLa})
\begin{equation*}
\sup_{\phi\in R\mc P_0^{m}}
\left( n_\Lambda(\Poi_{{\bC_\pm}} \phi)
-\mu_\sigma(\phi )
\right)	<+\infty .
\end{equation*}
This record is useful for comparison to \eqref{df:ord}.
\end{remark}

\section{On the completeness\\ of exponential systems}
\setcounter{equation}{0}
A system of vectors of vector topological space is
complete in  this space if the closure of its linear span coincides
with this space. Otherwise the system is
incomplete.

For spaces on a segment or an interval
on $\R$ traditionally consider exponential systems
 with  sequence of exponents
$$
i\Lambda :=\{ i \lambda_k\}_{k\in \N}, 
\quad \Exp^{i\Lambda}=\bigl\{e^{i\lambda x}\colon
\lambda\in \Lambda , \; x\in \R\bigr\}.
$$

Denote by $I_d\subset \R$ an arbitrary segment $[a,b]$
 of length $d=b-a$.

By virtue of known interrelation between
 uniqueness and completeness
the Main Theorem implies

\begin{theorem}[\rm
on completeness of exponential systems]\label{th:ces}
 If a point sequence $\Lambda=\{ \lambda_k\}_{k\in \N}\not\ni 0$ satisfies to the condition 
\begin{equation}\label{est:tcomp}
\sup_{\phi\in R\mc P_0^{m}}
\left( \sum_{k\in \N} (\mathrm P_{{\bC_\pm}} \phi)(\lambda_k)
-\frac{d}{2\pi}\int_{-\infty}^{+\infty}\phi (x) \d x
\right)	=+\infty ,
\end{equation}
then the system
$\Exp^{i\Lambda}$ is complete in
$C(I_{d})$ and~$L^p(I_{d})$.

Inversely, if 
\begin{equation}\label{est:tcomp-}
\sup_{\phi\in R\mc P_0^{m}}
\left( \sum_{k\in \N} (\Poi_{{\bC_\pm}} \phi)(\lambda_k)
-\frac{d}{2\pi}\int_{-\infty}^{+\infty}\phi (x) \d x
\right)	<+\infty ,
\end{equation}
then, for any pair different points
 $\{\lambda', \lambda''\}\subset \Lambda$,
the system
$\Exp^{i\Lambda\setminus\{i\lambda'\}}$
is incomplete in $C(I_{d})$ and
  $L^p(I_{d})$ for $p\geq 2$,
and the system
$\Exp^{i\Lambda\setminus
\{i\lambda',\, i\lambda''\}}$
is incomplete in $L^p(I_{d})$ for
$1\leq p<2$.
\end{theorem}
\begin{remark}
From Theorem \ref{th:ces} we can obtain  a whole number of basic old
results on completeness of exponential systems
in spaces $C(I_{d})$ and
 in $L^p(I_{d})$ (for example, the
Berling--Malliavin Theorem on radius of
 completeness),  and also new results.
\end{remark}

\begin{example} Let's prove very briefly and
 quickly one old result.

\begin{theorem}[{\rm \cite{Sch43}, \cite[Theorem 41]{Red77}}]\label{th:2} 
If a point sequence
$\Lambda=\{\lambda_k\}_{k\in \N}\subset \bC$ satisfies to two conditions
\begin{equation}\label{cond:SR}
0<\alpha \leq |\arg \lambda_k|\leq \pi-\alpha, 
\; k\in \N , \qquad
\sum_{k\in \N} \left|\Im \frac{1}{\lambda_k}
\right|<+\infty \, , 
\end{equation}
then the system $\Exp^{i\Lambda}$ is incomplete
in   any $C(I_{d})$ and
  $L^p(I_{d})$.
\end{theorem}
\begin{proof} Without loss generality we can 
 consider, that for arbitrary $\e$ the condition 
\begin{equation}\label{est:e}
\sum_{k\in \N} \left|\Im \frac{1}{\lambda_k} 
\right|<\e
\end{equation}
is fulfilled instead of second condition from \eqref{cond:SR}. For this purpose it is 
enough to shift finite number of points. 
It does not change the property of (in-)completeness 
(see Remark 2). First condition from \eqref{cond:SR} implies the
estimate $|\Im \lambda_k| \geq | \lambda_k|\sin \alpha $.
This estimate give the following estimate for
the Poisson kernel at every point $\lambda_k$: 
$$
\hspace{-1mm}\frac1{\pi}\,
\frac{|\Im \lambda_k|}{(t-\Re \lambda_k)^2
+(\Im \lambda_k)^2}\leq 
\frac{1}{\pi\sin^2 \alpha}\,
\frac{|\Im \lambda_k|}{| \lambda_k|^2}=
\frac{1}{\pi\sin^2 \alpha}\,
\left|\Im \frac{1}{\lambda_k}\right|.
$$
Hence in view  of \eqref{est:e} for arbitrary $\phi \in R\mc P_0^m$ we have
\begin{multline*}
\sum_{k\in \N} (\Poi_{{\bC_\pm}} \phi)(\lambda_k):=\sum_{k\in \N}\frac1{\pi}\int_{-\infty}^{+\infty}
\frac{|\Im \lambda_k|}{(t-\Re \lambda_k)^2
+(\Im \lambda_k)^2}\,\phi (t)\d t
\\
\leq
\sum_{k\in \N}\frac{1}{\pi\sin^2 \alpha}\,
\left|\Im \frac{1}{\lambda_k}\right|
\cdot \int_{-\infty}^{+\infty}\phi (t)\d t
\leq \frac{\e}{\pi\sin^2 \alpha}\,
\cdot \int_{-\infty}^{+\infty}\phi(x)\d x.
\end{multline*}
If we choose $\e$ sufficiently small, then 
$$
\frac{\e}{\pi\sin^2 \alpha}\leq \frac{d'}{2\pi}
<\frac{d}{2\pi}\, ,
$$
and the condition \eqref{est:tcomp-} of Theorem \ref{th:ces} is fulfilled with $d'$ instead of $d$.
By   Theorem \ref{th:ces} the exponential system $\Exp^{i\Lambda}$ is incomplete
in  spaces  $C(I_{d'})$ and
  $L^p(I_{d'})$ without one
 or two functions, where $d'<d$. Thereby,
the system $\Exp^{i\Lambda}$ is incomplete
in  spaces  $C(I_{d})$ and
 in $L^p(I_{d})$ for any $d>0$.
\end{proof}
\end{example}

Let's remind the Beurling--Malliavin Theorem in the Redheffer's interpretation (see \cite[Theorem 77]{Red77}, \cite{Koosis92}--\cite{HJ}).
\begin{theoBM}[on the radius of completeness] Let $\Lambda=\{\lambda_k\}_{k\in \N}\subset \bC$. If there exists a number $c>0$ and a sequence $\{n_k\}_{k\in \N}$ of distinct integers such that the series
\begin{equation}\label{df:red}
\sum_{k\in \N} \left|\frac{1}{\lambda_k}-\frac{c}{2\pi n_k} \right|	
\end{equation}
converges, then the system $\Exp^{i\Lambda}$
is incomplete in $C(I_d)$ and $L^p(I_d)$ for any $d>c$.
Inversely, if the series \eqref{df:red} diverges, then
the system $\Exp^{i\Lambda}$
is complete in $C(I_d)$ and $L^p(I_d)$ for any $d<c$.
\end{theoBM}
 
On the basis of Theorem \ref{th:ces}  on completeness of exponential systems it is possible to give a new proof of the Beurling--Malliavin Theorem on the radius of completeness. But it is impossible to name this new proof neither short, nor simple.

\section{New results}
Here we give only two results.

\begin{theorem} Let\/ $\Lambda=\{\lambda_k\}_{k\in \N}$ and\/ 
$ \Gamma =\{ \gamma_k\}_{k\in \N} \subset \mathbb C$  are two point sequences on\/ $\bC$ without limit points in\/ $\bC$.   Suppose that there is\/ $h\in \R$ such that  $\Im \lambda_k\neq -h$ and $\Im \gamma_k \neq -h$ for all $k\in \N$, and
{\small \begin{equation*}
\limsup_{t\to \pm\infty}
 \sum_{k\in \N} \left(\frac{|\Im \lambda_k+h|}{(t-\Re \lambda_k)^2
+(\Im \lambda_k+h)^2}
-\frac{|\Im \gamma_k+h|}{(t-\Re \gamma_k)^2
+(\Im \gamma_k+h)^2}
\right)	<+\infty .
\end{equation*}}
If the sequence\/ $\Gamma$ is  a zero sequence for\/ $B_\sigma^\infty$ , then the sequence\/ $\Lambda$ is also a zero sequence for\/ $B_\sigma^\infty$.

If the system\/ $\mathrm{Exp}^{i\Lambda}$ is complete in one of  spaces\/  $C(I_{d})$ or\/ $L^p(I_{d})$, $p\geq~1$, then 
for any\/  $\gamma', \gamma''\notin \Gamma$, $\gamma'\neq \gamma''$ the system\/ $\mathrm{Exp}^{i\Gamma\cup \{i\gamma'\}}$ is complete in\/ $C(I_{d})$ and in\/  $L^p(I_{d})$ for $p\geq 2$, and the system\/ $\mathrm{Exp}^{i\Gamma\cup
\{i\gamma',\, i\gamma''\}}$ is complete in\/  $L^p(I_{d})$ for\/ $1\leq p<2$.
\end{theorem}

The notation of excesses for exponential system $\Exp^{i\Lambda}$ see in \cite{Red77}. We denote its as $\exc i\Lambda$. 
\begin{corollary} Let $\Lambda \cap \R=\varnothing$
and $\Gamma \cap \R=\varnothing$. If
$$
\limsup_{t\to\pm\infty} \sum_{k\in \N}\left|\Im \frac{\lambda_k-\gamma_k}{(t-\lambda_k)(t-\gamma_k)}\right|<+\infty,
$$
then $\Lambda$ and $\Gamma$ can be  zero sequences for $B_\sigma^\infty$ only simultaneously. 

Besides, $|\exc i\Lambda-\exc i\Gamma|\leq 1$ for $C(I_{d})$ and $L^p(I_{d})$ with $p\geq 2$, and $|\exc i\Lambda-\exc i\Gamma|\leq 2$ for $L^p(I_{d})$ with $1\leq p < 2$.
\end{corollary}

\medskip

E-mail: {\tt khabib-bulat@mail.ru}

 Web-site: {\tt http://math.bsunet.ru/khb{\_}e}

Bashkir State University, Ufa, Bashkortostan, RUSSIA

\end{document}